\documentclass[11pt, twoside]{amsart}
\usepackage[utf8]{inputenc}
\usepackage{color}
\usepackage{amscd,amsfonts,amsmath,amssymb,amsthm,amsxtra,amstext,mathtools,latexsym,bbm,hyperref,enumitem,indentfirst,url}
\usepackage{pifont}
\usepackage{graphicx}

\title{Constant Mean Curvature Surfaces for the Bessel Equation}

\author{E. Mota}
\address{Eduardo Mota, Department of Mathematics, University College Cork, Ireland.}
\email{eduardo.motasanchez@ucc.ie}


\theoremstyle{plain}
\newtheorem{proposition}{Proposition}
\newtheorem{theorem}{Theorem}
\newtheorem{lemma}{Lemma}
\newtheorem{definition}{Definition}

\newtheorem{remark}{Remark}

\numberwithin{equation}{section}

\newcommand{\Z}{\mathbb{Z}}
\newcommand{\R}{\mathbb{R}}
\newcommand{\Sp}{\mathbb{S}}
\newcommand{\C}{\mathbb{C}}

\newcommand{\bigO}{\mathcal{O}}
\newcommand{\Iso}{\mathrm{Iso}}

\newcommand{\su}{\mathfrak{su}_2}
\newcommand{\sll}{\mathfrak{sl}_2}

\newcommand{\SL}{\mathrm{SL}_2}
\newcommand{\SU}{\mathrm{SU}_2}

\newcommand{\CMC}{$\mathrm{CMC}\;$}
\newcommand{\ODE}{$\mathrm{ODE}\;$}
\newcommand{\dd}{\mathrm{d}}
\newcommand{\Sym}{\mathrm{Sym}}
\newcommand{\identity}{\mathbbm{1}}

\DeclareMathOperator{\tr}{tr}
\DeclareMathOperator{\diag}{diag}

\DeclareMathOperator{\id}{id}
\DeclareMathSymbol{\Nu}{\mathalpha}{operators}{"4E}

\DeclareMathAlphabet{\mathcal}{OMS}{cmsy}{m}{n}
\SetMathAlphabet{\mathcal}{bold}{OMS}{cmsy}{b}{n}

\hypersetup{
    bookmarks=true,         
    unicode=false,          
    pdftoolbar=true,        
    pdfmenubar=true,        
    pdffitwindow=false,     
    pdfstartview={FitH},    
    pdfnewwindow=true,      
    colorlinks=false,       
    linkcolor=red,          
    citecolor=green,        
    filecolor=magenta,      
    urlcolor=cyan           
}

\usepackage[noabbrev]{cleveref}

\begin{document}

  \begin{abstract}
    In this note we construct a family of immersions with constant mean curvature of the twice-punctured Riemann sphere into $\R^3$ from the Bessel equation.
  \end{abstract}

  \maketitle

  \section*{Introduction}

    The generalised Weierstrass representation due to Dorfmeister, Pedit and Wu \cite{DPW} can be used to construct locally conformal constant mean curvature ($\mathrm{CMC}$) immersions in Euclidean $3$-space from a holomorphic 1-form $\xi$ on a Riemann surface $\Sigma$. Several examples of \CMC surfaces have been made using these techniques by Dorfmeister, Wu \cite{DorWu}, Haak \cite{DorH:cyl}, Kilian, McIntosh, Schmitt \cite{KIS}, Kobayashi, Rossmann \cite{KilKRS} and Traizet \cite{Tr:noids}, among others. Also, more recently \cite{Tr:nodes}, Traizet uses an \emph{opening nodes} method in the generalised Weierstrass representation set-up to obtain embedded \CMC surfaces with Delaunay ends, with arbitrary genus and number of ends.\\

    In the spirit of \cite{KMS}, our aim is to employ a second order differential equation, the Bessel equation, to find a new family of cylinders with constant mean curvature. This family possess one asymptotically Delaunay end and one irregular end, corresponding respectively to the regular and irregular singularities in the $\mathrm{ODE}$. The first step in the construction of \CMC surfaces is to write down a suitable potential $\xi$. That one of the ends of the surfaces is asymptotic to half Delaunay surface follows from the fact that at this end $\xi$ is a perturbation of the potential  of a Delaunay surface \cite{KilRS}. Since $\Sigma$ is the twice-punctured Riemann sphere, it is enough to guarantee closing conditions at the Delaunay end in order to solve the period problem.\\

    In \cref{sec:gwr} we outline the constructing method for \CMC surfaces, due to \cite{DPW}. The \cref{sec:BE} has as purpose introducing the Bessel equation in our recipe for surfaces with constant mean curvature, along with a brief explanation of its features. \Cref{sec:perturbations} recalls the very well-known theory regarding \CMC surfaces with Delaunay ends and perturbations of Delaunay potentials (see \cite{KilRS} and \cite{KilKRS}, among others). In the main section in this note, \ref{sec:cylinders}, we write down the potential that constructs \CMC cylinders via the Bessel equation and put together in \cref{th:maintheorem} the different ingredients for the construction. Finally, in \cref{sec:symmetries} we prove the existence of a reflectional symmetry in the \CMC cylinders that we have constructed. Graphics in figures \ref{fig:examples1} and \ref{fig:examples2} were produced with CMCLab \cite{Sch}.

  %

  \section{The generalised Weierstrass representation}\label{sec:gwr}

    Let us briefly recall the generalized Weierstrass representation \cite{DPW} to set the notation. We refer the reader to \cite{KilKRS} for details on this method, which comprises the following steps:

    \begin{enumerate}
      \item On a connected Riemann surface $\Sigma$, let $\xi$ be a holomorphic 1-form, called \emph{potential}, with values in the loop algebra of maps $\Sp^1\to\sll(\C)$. The potential $\xi$  has to have a simple pole in its upper right entry in the loop parameter $\lambda$ at $\lambda=0$, and has no other poles for $\lambda<1$. Moreover, the upper-right entry of $\xi$ is non-zero on $\Sigma$.  Let $\Phi$ be a solution of
      \begin{equation}\label{eq:IVP}
        \begin{cases}
          \dd\Phi = \Phi\,\xi\\
          \Phi(z_0) = \Phi_0.
        \end{cases}
      \end{equation}
      \item Let $\Phi=F\,B$ be the point-wise Iwasawa factorization on the universal cover $\widetilde\Sigma$.
      \item Then, the Sym-Bobenko formula $f:=\Sym[F_\lambda]=\left(\partial_\lambda F\right)F^{-1}$ gives an associated family of conformal \CMC immersions $\widetilde\Sigma \to \su\cong\R^3$.
    \end{enumerate}
        
    If $g$ is a map on $\Sigma$ with values in a positive loop group of $\SL(\C)$, then $\xi.g$ is again a potential. The gauge action is defined by $\xi.g:=g^{-1}\xi g+g^{-1}\dd g$. If $\dd\Phi = \Phi\,\xi$, then $\Phi g$ solves $\dd\Psi=\Psi(\xi.g)$.\\

    Now let $\Sigma = \C^*$ be the twice-punctured Riemann sphere. Let $\Delta$ denote the group of deck transformations of the universal cover $\widetilde\Sigma$, that is, $\Delta=\Z$ and $\widetilde\Sigma=\C$. The group of deck transformations $\Delta$ is generated by
    \begin{equation}\label{eq:tau}
      \tau:\log z\mapsto\log z+2\pi i.
    \end{equation}
    Let $\xi$ be a holomorphic potential on $\Sigma$. Let $\Phi$ be a solution of $\dd\Phi = \Phi \xi$. Since $\Sigma$ is not simply connected, in general $\Phi$ is only defined on $\widetilde\Sigma$. Let $\tau$ be a loop around the puncture $z=0$, we define the \emph{monodromy matrix} $M$ of $\Phi$ along $\tau$ by
    \begin{equation}\label{eq:defmonodromy}
      M = (\tau^*\Phi) \, \Phi^{-1}.
    \end{equation}
    In order to construct \CMC cylinders, the sufficient closing conditions for the monodromy that ensure that the immersion is well-defined are (see \cite{KilKRS}):
    \begin{subequations}\label{eq:monodromyproblem}
      \begin{align}
        M &\in\SU,\;\text{for all}\;\lambda\in\Sp^1,\label{eq:monodromyproblem1}\\
        \left. {M}\right|_{\lambda = 1} &= \pm\identity,\label{eq:monodromyproblem2}\\
        \left. \partial_\lambda M \right|_{\lambda = 1} &= 0.\label{eq:monodromyproblem3}
      \end{align}
    \end{subequations}
    Since $\Sigma=\C^*$, the monodromy group is infinite cyclic. Thus the fundamental group of the twice-punctured Riemann sphere has $1$ generator and we only need to solve the monodromy problem (\ref{eq:monodromyproblem}) for the loop $\tau$ around the puncture $z=0$.
  %

  \section{The Bessel equation}\label{sec:BE}

    In this section we show how the Bessel equation constructs \CMC surfaces with two ends. We start by prescribing this \ODE in the first step (\ref{eq:IVP}) of the Weierstrass recipe \cite{DPW} and then we point out the relevant features of this differential equation for the purpose of our work, namely, we describe its singularities' behaviour.\\

    Let us consider without loss of generality a potential $\xi$ of the form
    \begin{equation}\label{eq:offdpot}
      \xi=\begin{pmatrix} 0 & \nu(z,\lambda)\\ \rho(z,\lambda) & 0\end{pmatrix}\;dz.
    \end{equation}
    The strategy to associate a scalar second order \ODE to our constructing algorithm is illustrated in the following
    \begin{lemma}\label{th:associatedode}
        Solutions of $\dd\Phi=\Phi\xi$ are of the form
        \begin{equation}\label{eq:generalsolution}
        \begin{pmatrix} y_1'/\nu & y_1 \\ y_2'/\nu & y_2 \end{pmatrix}
        \end{equation}
        where $y_1$ and $y_2$ are a fundamental system of the scalar \ODE
        \begin{equation}\label{eq:scalar2ndode}
          y'' - \frac{\nu'}{\nu}\, y' - \rho\,\nu\,y = 0.
        \end{equation}
    \end{lemma}
    By means of \cref{th:associatedode}, we can choose functions $\nu$ and $\rho$ in the potential (\ref{eq:offdpot}) so that the Bessel equation becomes the associated \ODE (\ref{eq:scalar2ndode}) in the initial value problem (\ref{eq:IVP}).\\

    The \emph{Bessel equation} \cite{Bow} is the linear second-order \ODE given by
    \begin{equation}\label{eq:BE1}
      z^2 y''+z y'+(z^2-\alpha^2)y=0.
    \end{equation}
    Equivalently, dividing through by $z^2$,
    \begin{equation}\label{eq:BE2}
      y''+\frac{1}{z} y'+\left(1-\frac{\alpha^2}{z^2}\right)y=0.
    \end{equation}
    The paramater $\alpha$ is an arbitrary complex number. This equation has one regular singularity at $z=0$ and one irregular singularity of rank $2$ at $\infty$ (see \cite{Kris}), and for this reason the Bessel equation does not belong to the class of \emph{Fuchsian equations}. Each of these punctures will correspond to one end on the \CMC surface; the regular singularity generates a \emph{regular} end (asymptotic to a Delaunay surface) while the irregular singularity generates an \emph{irregular} end.\\

    Consider again the off-diagonal potential in (\ref{eq:offdpot}) and let us choose functions
    \begin{equation}\label{eq:functionspotential}
      \nu:=\frac{1}{z}\quad\text{and}\quad\rho:=-z + \frac{\alpha^2}{z}.
    \end{equation}
    Plugging $\nu$ and $\rho$ into \cref{eq:scalar2ndode}, one obtains the Bessel equation and, in particular, prescribes this \ODE in the algorithm for \CMC surfaces seen in \cref{sec:gwr}.

  %

  \section{The $z^AP$ Lemma and perturbations at a simple pole}\label{sec:perturbations}

    Consider a differential equation $\dd\Phi=\Phi\xi$ for which the potential can be written as $\xi=A{dz\over z-z_k}+\bigO(z^0)\;dz$, that is, it has a simple pole at $z=z_k$. A basic result in \ODE theory can be extended to the context of loops. This is known as the $z^AP$ \emph{lemma} (see \cite[Lemma 2.3]{KilRS}) and states that under certain conditions on the eigenvalues of $A$, there exists a solution of the form $\Phi=Cz^AP=C\exp\left(A\log{z}\right)P$, where $P$ extends holomorphically to $z=z_k$ and $C$ is in the loop group of $\SL(\C)$.
    \begin{lemma}\cite[Lemma 14]{KilKRS}\label{th:zap}
      Let $A:\C^*\to\sll(\C)$ be an analytic map with non-constant eigenvalues $\pm\mu(\lambda)$. Let $\xi=A(\lambda)dz/z+\bigO(z^0)dz$ in a neighbourhood of $z=0$.\\
      Let $M:\C^*\to\SL(\C)$ be a monodromy of $\xi$ associated to a once-wrapped closed curve around $z=0$. Then $\tr M= 2\cos(2\pi\mu)$ on $\C^*$.
    \end{lemma}
    A \emph{Delaunay residue} is a meromorphic $\sll(\C)$-valued map of the form
    \begin{equation}\label{eq:residue}
      A=\begin{pmatrix} c & a\lambda^{-1}+\bar{b}\\ \bar{a}\lambda+b & -c\end{pmatrix},
    \end{equation}
    with $a,b\in\C^*$ and $c\in\R$. Any Delaunay surface in $\R^3$, up to rigid motion, can be derived \cite{KIS} from an off-diagonal Delaunay residue (\ref{eq:residue}) with $a,b\in\R^*$ satisfying the closing condition $a+b=1/2$, so that the resulting surface's \emph{necksize} depends on $ab$: when $ab>0$, the surface is an \emph{unduloid}, when $ab < 0$, it is a \emph{nodoid}, and when $a = b$, it is a round cylinder (see for instance \cite{KilRS}). On the other hand, a \emph{perturbation} of a Delaunay potential is a potential of the form
    \begin{equation}\label{eq:perturbation}
      \xi=A\frac{dz}{z}+\bigO(z^0)dz
    \end{equation}
    where $A$ is as in (\ref{eq:residue}). That a \CMC end obtained from a perturbation of a Delaunay potential is asymptotic to a half-Delaunay surface is proven in \cite[Theorem 5.9]{KilRS}. We will use this result in the next section.
  %

  \section{Constructing $\mathrm{CMC}\;$ cylinders}\label{sec:cylinders}

    Constructing \CMC surfaces with two ends via the Bessel equation is in the following two steps:
    \begin{itemize}
      \item Write down a potential on $\C^*$ which prescribes the Bessel equation as associated \ODE and which is locally gauge-equivalent to a perturbation of a Delaunay potential at $z=0$ (\cref{def:potential}).
      \item Show that the monodromy representation $M$ is unitary in $\Sp^1$ (\cref{sec:monodromy}).
    \end{itemize}

    \subsection{Cylinder potential}\label{sec:potential}
      In this part the potential which will be used to produce cylinders with one irregular end and one asymptotic Delaunay end is defined. Near the puncture $z=0$ the potential is a local perturbation of a Delaunay potential via gauge equivalence.
      \begin{definition}\label{def:potential}
        Let $\Sigma=\C^*$ and let $r\in(-\infty,1)\setminus\lbrace 0\rbrace$. Define the $\sll(\C)$-valued cylinder potential by
        \begin{equation}\label{eq:xiC}
          \xi_c=\begin{pmatrix} 0 & \lambda^{-1}\\ \lambda Q_t & 0\end{pmatrix}\;dz,
        \end{equation}
        where $Q_t=\left(-\frac{r}{4 z^2}t-1\right)$ and $t:=-{1\over 4}\lambda^{-1}(\lambda-1)^2$, for all $\lambda\in\Sp^1$.
      \end{definition}
      \begin{remark}
        Note that if $\lambda=1$, then $t=0$ and the potential (\ref{eq:xiC}) becomes holomorphic so the monodromy $M(\lambda=1)=\identity$. Hence, \cref{eq:monodromyproblem2} and \cref{eq:monodromyproblem3} of the monodromy problem are automatically solved.
      \end{remark}
      %
    %

    \subsection{Local gauge}\label{sec:gauge}
      In this part we use the theory of gauging to see that the potential in \cref{sec:BE} with associated \ODE the Bessel equation, is equivalent to the potential defined in \ref{def:potential}. Then, it is shown that the double pole of the potential in (\ref{eq:xiC}) can be gauged to a simple pole with Delaunay residue.\\

      Consider the potential $\xi$ introduced in \cref{sec:BE}, namely
      \begin{equation}\label{eq:firstxi}
        \xi=\begin{pmatrix} 0 & 1/z\\ -z + \alpha^2/z & 0\end{pmatrix}\;dz.
      \end{equation}
      Note that for
      \begin{equation}\label{eq:gauges1}
        g_1=\begin{pmatrix} (1/z)^{1/2} & 0\\ 0 & (1/z)^{-1/2}\end{pmatrix},\quad g_2=\begin{pmatrix} 1 & 0\\ \frac{1}{2z} & 1\end{pmatrix}.
      \end{equation}
      the gauge $\xi.(g_1g_2)$ gives
      \begin{equation}\label{eq:secondxi}
        \xi=\begin{pmatrix} 0 & 1\\ -1 + \frac{-1 + 4 \alpha^2}{4 z^2} & 0\end{pmatrix}\;dz.
      \end{equation}
      Putting $\alpha:=\frac{1}{2}\sqrt{1 - r t}$ (note that the parameter $\alpha$ depends on $\lambda$) and gauging by $\Lambda=\diag(\lambda^{1/2},\lambda^{-1/2})$ we obtain the constructing potential $\xi_c$ from \cref{def:potential}.
      \begin{lemma}\label{th:gauging}
        Let $\xi_c$ a potential as in (\ref{eq:xiC}). Then there exists a neighbourhood $U$ of $z=0$ and a positive gauge $g$ such that the expansion of $\xi_c.g$ is
        \begin{equation}\label{eq:xigauged}
          A\frac{dz}{z}+\bigO(z^0)dz,\quad\text{where}\;A=\begin{pmatrix} 0 & a\lambda^{-1}+b\\ a\lambda+b & 0\end{pmatrix}
        \end{equation}
        for some $a,b\in\R$ with $a+b=1/2$.
      \end{lemma}
      \begin{proof}
        Let
        \begin{equation}\label{eq:gauges2}
          g_{1c}=\begin{pmatrix} z^{1/2} & 0\\ 0 & z^{-1/2}\end{pmatrix},\quad g_{2c}=\begin{pmatrix} 1 & 0\\ -\frac{1}{2}\lambda & a+b\lambda\end{pmatrix}.
        \end{equation}
        Taking the real values $a=\frac{1}{4}\left(1+\sqrt{1-r}\right)$ and $b=\frac{1}{4}\left(1-\sqrt{1-r}\right)$, then $\xi_c.(g_{1c}g_{2c})$ has a simple pole at $z=0$ and is of the form (\ref{eq:xigauged}). Therefore, $g=g_{1c}g_{2c}$ is the required gauge.
      \end{proof}
      %
    %

    \subsection{Unitary monodromy on the twice-puncture Riemann sphere}\label{sec:monodromy}

      Since the fundamental group of the twice-punctured Riemann sphere has only $1$ generator, we just need to solve the monodromy problem (\ref{eq:monodromyproblem}) at $z=0$. Recall that, with our choice of potential (\ref{eq:xiC}), (\ref{eq:monodromyproblem2}) and (\ref{eq:monodromyproblem3}) of the monodromy problem are already solved.\\

      For $\Phi_0=\identity$ and $z_0=1$, consider a solution $\Phi$ of (\ref{eq:IVP}) for $\xi=A$ with $A$ as in \cref{th:gauging}, that is, satisfying that $a,b\in\R$ and $a+b=1/2$. Then $\Phi=z^A$ and the monodromy is given by
      \begin{equation}\label{eq:M}
        M=\exp(2\pi iA)=\cos(2\pi\mu)\identity+\frac{1}{\mu}\sin(2\pi\mu)A,
      \end{equation}
      where it holds for the eigenvalues that
      \begin{equation}\label{eq:mu}
        \mu^2=a^2+b^2+ab(\lambda^{-1}+\lambda).
      \end{equation}
      Thus, since $a+b=1/2$, $M$ is unitary for all $\lambda\in\Sp^1$ and then also (\ref{eq:monodromyproblem1}) is solved. The immersion will be well defined.
    %

    %
    \begin{figure}[t] 
      \centering
      \includegraphics[width=5.5cm]{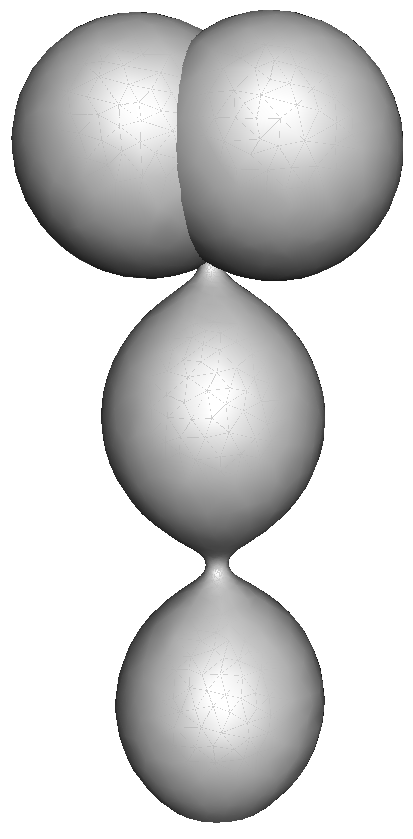}
      \includegraphics[width=7cm]{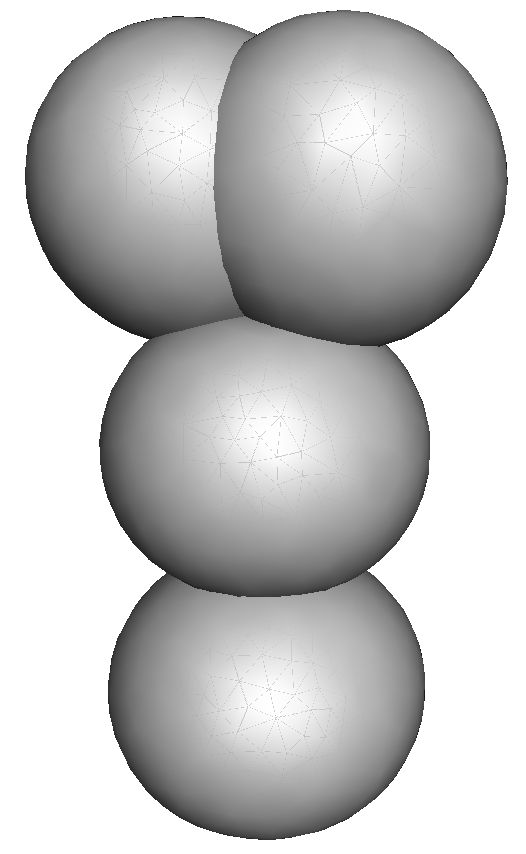}
      \caption{\label{fig:examples1} \footnotesize Cylinders with one irregular end and one Delaunay end. The regular end weights are $1/3$ and $-1/4$, so the surfaces have one unduloid end and one nodoid end respectively.}
    \end{figure}
    \begin{figure}[t] 
      \centering
      \includegraphics[width=4cm]{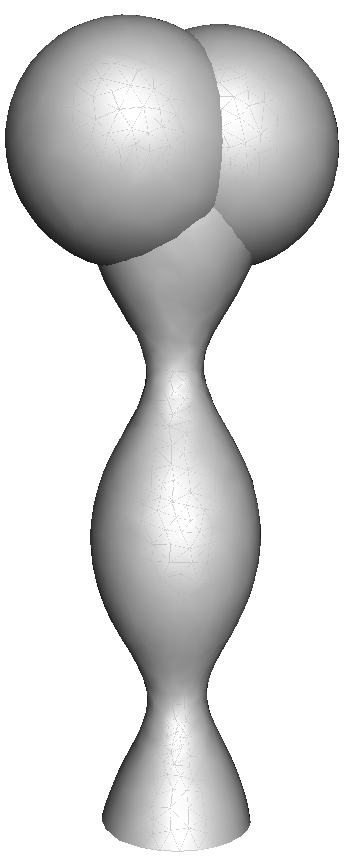}
      \includegraphics[width=8.5cm]{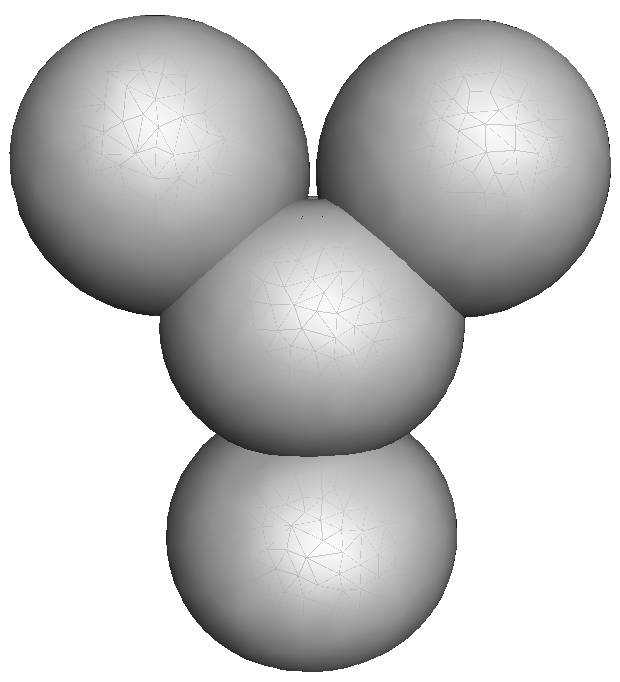}
      \caption{\label{fig:examples2} \footnotesize Cylinders with one irregular end and one Delaunay end. The regular end weights are $1/\sqrt{2}$ and $-1/\pi$, so the surfaces have one unduloid end and one nodoid end respectively.}
    \end{figure}
    %
    
    \subsection{Main theorem}\label{sec:finaltheorem}
      \begin{theorem}\label{th:maintheorem}
        Let $\Sigma=\C^*$ and let $r\in(-\infty,1)\setminus\lbrace 0\rbrace$. Then, there exists a conformal \CMC immersion $f:\Sigma\to\R^3$ with one end which is asymptotic to half Delaunay surface and one irregular end.
      \end{theorem}
      \begin{proof}
        Let $\xi_c$ be a cylinder potential as in \cref{def:potential}. A solution $\Phi$ of (\ref{eq:IVP}) for $\xi_c$, $\Phi_0=\identity$ and $z_0=1$ can be found with the $z^AP$ lemma \cite{KilKRS}. Let $M$ be the monodromy of $\Phi$ at the puncture $z=0$. By the remarks in \cref{sec:monodromy}, $M$ is unitary and the monodromy problem (\ref{eq:monodromyproblem}) at $z=0$ is solved. Then, the general Weierstrass representation \cite{DPW} constructs a \CMC immersion $f$ in $\R^3$ which has two ends corresponding to the singularities from the Bessel equation - the \ODE associated to $\dd\Phi=\Phi\xi_c$. By \cref{th:gauging}, the potential is locally gauge-equivalent to a Delaunay potential and thus, by the asymptotics theorem of \cite{KilRS}, the end at $z=0$ is asymptotic to half Delaunay surface.
      \end{proof}
      %
    %

  \section{Symmetry}\label{sec:symmetries}

    In this last part, we explore some symmetries appearing in the surfaces constructed in \cref{sec:cylinders}. We prove that these symmetries in the resultant surfaces can be tracked to the level of the potentials, which are transformed under automorphisms on the domain.\\

    A transformation $\phi$ is an involution if $\phi^2=\id$ but $\phi\neq\id$. The following proposition holds for elements of $\Iso(\R^3)$ that are involutions.
    \begin{proposition}\label{th:involutions}
      The involutory isometries are the reflections and the half-turns (rotations by $\pi$).
    \end{proposition}
    In what follows, to lighten notation, we may denote the dependence on $\lambda$ with a subscript, that is,
    \begin{equation}
      \begin{aligned}
        \xi_\lambda=\xi(z,\lambda),\\
        \Phi_\lambda=\Phi(z,\lambda),\\
        F_\lambda=F(z,\lambda).
      \end{aligned}
    \end{equation}

    \subsection{Cylinders with one reflection}\label{sec:reflection}

      Consider the orientation reversing automorphism of $\Sigma$ given by $\sigma(z)=\bar{z}$. It defines a symmetry that reflects the domain across the real axis. We prove the following
      \begin{theorem}\label{th:reflectionsymmetry}
        Consider a potential $\xi$ that generates via the Weierstrass representation a \CMC family of immersions $f$. Suppose that $\xi$ satisfies the symmetries
        \begin{equation}\label{eq:reflectionsymmetrypotentials}
          \begin{aligned}
            \xi_{1/\lambda} &=\overline{\sigma^*\xi_{1/\bar{\lambda}}},\\
            G^{-1}\xi_\lambda G &=\overline{\sigma^*\xi_{1/\bar{\lambda}}},
          \end{aligned}
        \end{equation}
        where $G=\diag(1/\lambda,\lambda)$. Then, the induced immersion $\breve{f}=\Sym\left[\overline{\sigma^*F_{1/\bar{\lambda}}}\right]$ possesses reflective symmetry by a plane.
      \end{theorem}
      \begin{proof}
        Let $\Phi_\lambda$ be the solution of $\dd\Phi_\lambda=\Phi_\lambda\xi_\lambda$, $\Phi_\lambda(z_0)=\Phi_0$, with $z_0\in\Sigma$ and $\Phi_0$ diagonal. Naturally, the transformation $\overline{\sigma^*\Phi_{1/\bar{\lambda}}}$ defines a solution to the differential equation $\;\dd\left(\overline{\sigma^*\Phi_{1/\bar{\lambda}}}\right)=\left(\overline{\sigma^*\Phi_{1/\bar{\lambda}}}\right)\left(\overline{\sigma^*\xi_{1/\bar{\lambda}}}\right)$, which in view of (\ref{eq:reflectionsymmetrypotentials}) reads as
        \begin{equation}
          \dd\left(\overline{\sigma^*\Phi_{1/\bar{\lambda}}}\right)=\left(\overline{\sigma^*\Phi_{1/\bar{\lambda}}}\right)\left(\xi_\lambda\,.\,G\right).
        \end{equation}
        Since this \ODE is also solved by $\Phi_\lambda G$, i.e.
        \begin{equation}
          \dd\left(\Phi_\lambda G\right)=\left(\Phi_\lambda G\right)\left(\xi_\lambda\,.\,G\right),
        \end{equation}
        both solutions only differ by a matrix $R$ in the loop group of $\SL(\C)$. Hence, $\Phi_\lambda$ has the symmetry
        \begin{equation}\label{eq:reflectionsolutionsymmetry2}
          R\Phi_\lambda=\overline{\sigma^*\Phi_{1/\bar{\lambda}}}G^{-1},
        \end{equation}
        for some $z$-independent $R$.\\

        Evaluation at the fixed point $1$ of $\sigma$, using \cref{eq:reflectionsolutionsymmetry2}, yields $R=\overline{\Phi_0(1/\bar{\lambda})}G^{-1}\Phi_0(\lambda)^{-1}$. Since $\Phi_0$ is diagonal, then by a simple calculation one gets that $R$ is unitary for all $\lambda\in\Sp^1$.\\

        Let us write the Iwasawa splitting $\Phi_\lambda=FB$. When Iwasawa decomposing the solution $\overline{\sigma^*\Phi_{1/\bar{\lambda}}}$ a unitary term $U$ must be introduced, obtaining
        \begin{equation}\label{eq:reflectioniwasawasymmetry}
          \overline{\sigma^*F_{1/\bar{\lambda}}}\,\overline{\sigma^*B_{1/\bar{\lambda}}}=\overline{\sigma^*\Phi_{1/\bar{\lambda}}}=R\Phi_\lambda G=RFUU^{-1}BG.
        \end{equation}
        The uniqueness of this splitting allows us to identify unitary and positive parts respectively, obtaining that
        \begin{equation}\label{eq:reflectionframesymmetry}
          \overline{\sigma^*F_{1/\bar{\lambda}}}=RFU.
        \end{equation}
        This implies that, using the generalised Weierstrass representation, $\overline{\sigma^*\xi_{1/\bar{\lambda}}}$ produces on the one hand the family of immersions given by plugging $\overline{\sigma^*F_{1/\bar{\lambda}}}$ in the Sym-Bobenko formula and on the other hand the one obtained using $RFU$. Consequently, these two surfaces coincide.\\

        At the level of the immersion, the symmetry (\ref{eq:reflectionframesymmetry}) of the unitary frame appears in the Sym-Bobenko formula $\breve{f}=\Sym\left[\overline{\sigma^*F_{1/\bar{\lambda}}}\right]$ as follows:
        \begin{equation}\label{eq:reflectionimmersionsymmetry}
          \begin{aligned}
            \breve{f} &=\left(\partial_\lambda \overline{\sigma^* F_{1/\bar{\lambda}}}\right)\, \left(\overline{\sigma^* F_{1/\bar{\lambda}}}\right)^{-1}\\
            &=\partial_\lambda (RFU)\, \left(RFU\right)^{-1}\\
            &=\left((\partial_\lambda R)R^{-1}+R(\partial_\lambda F)F^{-1}R^{-1}\right).
          \end{aligned}
        \end{equation}

        It is left to prove that this symmetry is a reflection. To do so, we show that the transformation is an involution. It is easy to check that each of the symmetries seen so far remain the same if we `reapply' the transformations done in (\ref{eq:reflectionsymmetrypotentials}) and used throughout this proof. Let us denote by $\breve{\breve{f}}=\Sym\left[\overline{\sigma^*\left(\overline{\sigma^*F_{1/\bar{\lambda}}}\right)_{1/\bar{\lambda}}}\right]$ the resultant immersion of applying again those transformations, and by $'$ the derivative with respect to $\partial_\lambda$. Then,
        \begin{equation}\label{eq:immersioninvolution}
          \begin{aligned}
            \breve{\breve{f}} &=\overline{\sigma^*R_{1/\bar{\lambda}}}\;\partial_\lambda\left(\overline{\sigma^*F_{1/\bar{\lambda}}}\right)\overline{\sigma^*F_{1/\bar{\lambda}}^{-1}}\overline{\sigma^*R_{1/\bar{\lambda}}^{-1}}\\
            &=\bar{R}_{1/\bar{\lambda}}\left(R'R^{-1}+RF'F^{-1}R^{-1}\right)\bar{R}_{1/\bar{\lambda}}^{-1}+\bar{R'}_{1/\bar{\lambda}}\bar{R}_{1/\bar{\lambda}}^{-1}\\
            &=\bar{R}_{1/\bar{\lambda}}RF'F^{-1}R^{-1}\bar{R}_{1/\bar{\lambda}}^{-1}+\bar{R}_{1/\bar{\lambda}}R'R^{-1}\bar{R}_{1/\bar{\lambda}}^{-1}+\bar{R'}_{1/\bar{\lambda}}\bar{R}_{1/\bar{\lambda}}^{-1}\\
          \end{aligned}
        \end{equation}
        An easy computation yields that $\bar{R}_{1/\bar{\lambda}}R=\identity$. Hence, derivating in this equation, one also gets that
        \begin{equation}\label{eq:Rscancel}
          \bar{R}_{1/\bar{\lambda}}R'R^{-1}\bar{R}_{1/\bar{\lambda}}^{-1}=-\bar{R'}_{1/\bar{\lambda}}\bar{R}_{1/\bar{\lambda}}^{-1}.
        \end{equation}
        Putting together \cref{eq:immersioninvolution} and \cref{eq:Rscancel}, we manage to derive that $\breve{\breve{f}}=\left(\partial_\lambda F\right)F^{-1}=\Sym[F_\lambda]=f$. That is, this symmetry is an involution. Since, it is also orientation reversing, by \cref{th:involutions}, this symmetry must be a reflection.
      \end{proof}

      Using \cref{th:reflectionsymmetry}, since the constructing potential (\ref{eq:xiC}) satisfies the relations in (\ref{eq:reflectionsymmetrypotentials}), the resulting surfaces in figures \ref{fig:examples1} and \ref{fig:examples2} have a reflectional symmetry fixing their ends.
    %

  %

  \bibliographystyle{amsplain}
  \providecommand{\bysame}{\leavevmode\hbox to3em{\hrulefill}\thinspace}
  \providecommand{\MR}{\relax\ifhmode\unskip\space\fi MR }
    \providecommand{\MRhref}[2]{%
      \href{http://www.ams.org/mathscinet-getitem?mr=#1}{#2}
    }
    \providecommand{\href}[2]{#2}

\end{document}